\documentclass[12pt]{amsart}
\usepackage{amssymb,setspace,mathrsfs}
\usepackage[author-year]{amsrefs}

\title[Dynamical Percolation]{Dynamical Percolation on General Trees}
   \thanks{Research supported in part by a grant from the National Science
   Foundation}
   \author[D. Khoshnevisan]{Davar Khoshnevisan}
   \address{Davar Khoshnevisan:\ Department\@ of Mathematics,
	   The University of Utah,
      155 S.\@ 1400 E.\@ Salt Lake City, UT 84112--0090}
   \email{davar@math.utah.edu}
   \urladdr{http://www.math.utah.edu/\~{}davar}
\keywords{Dynamical percolation; capacity; trees}
\subjclass{Primary. 60K35; Secondary. 31C15, 60J45}
\date{First draft: May 5, 2006; Final draft:
	January 17, 2007}

\theoremstyle{plain}{
\newtheorem{theorem}{Theorem}[section]}
\theoremstyle{plain}{
\newtheorem{proposition}[theorem]{Proposition}}
\theoremstyle{plain}{
   \newtheorem{lemma}[theorem]{Lemma}}
\theoremstyle{plain}{
   \newtheorem{corollary}[theorem]{Corollary}}
\theoremstyle{definition}{
   }
\theoremstyle{definition}{
   \newtheorem{example}[theorem]{Example}}
\theoremstyle{remark}{
   \newtheorem{remark}[theorem]{Remark}}
\theoremstyle{remark}{
   }
\numberwithin{equation}{section}
\newcommand{\dimh}{\dim_{_{\rm H}}}

\newcommand{\dimp}{\dim_{_{\rm P}}}
\newcommand{\1}{\mathbf{1}}
\renewcommand{\P}{\mathrm{P}}

\newcommand{\E}{\mathrm{E}}
\newcommand{\R}{\mathbf{R}}
\newcommand{\e}{\varepsilon}
\renewcommand{\mathcal}{\mathscr}
\begin{document}
\onehalfspacing
\begin{abstract}
	H\"aggstr\"om, Peres, and Steif \ycite{HaggstromEtAl}
	have introduced a dynamical version of percolation
	on a graph $G$. When $G$ is a tree they derived a necessary and
	sufficient condition for percolation to exist at
	some time $t$. In the case that $G$ is a spherically symmetric
	tree, \ocite{PeresSteif} derived a necessary and sufficient
	condition for percolation to exist at some time
	$t$ in a given target set $D$. The main result of the present paper is a
	necessary and sufficient condition for the existence of
	percolation, at some time $t\in D$, in the case
	that the underlying tree is not necessary spherically symmetric.
	This answers a question of Yuval Peres (personal communication).
	We present also a formula for the Hausdorff dimension of the
	set of exceptional times of percolation. 
\end{abstract}
\maketitle

\section{\bf Introduction}

Let $G$ be a locally finite connected graph.
Choose and fix $p\in(0\,,1)$, and let each edge 
be ``open'' or ``closed'' with respective probabilities
$p$ and $q:=1-p$; all edge assignments are made independently.
Define $\P_p$ to be the
resulting product measure on the collection of random
edge assignments.

A fundamental problem of bond percolation
is to decide when there can exist an infinite connected
cluster of open edges in $G$ \cite{Grimmett}. Choose and fix
some vertex $\rho$ in $G$, and consider the event
$\{\rho\leftrightarrow\infty\}$ that percolation occurs
through $\rho$. 
That is, let
$\{\rho\leftrightarrow\infty\}$ denote the event that
there exists an
infinite connected cluster of open edges that emanate
from the vertex $\rho$. 
Then the stated problem of percolation theory 
is, when is $\P_p\{\rho\leftrightarrow\infty\}>0$?
We may note that the positivity of $\P_p\{\rho\leftrightarrow\infty\}$
does not depend on our choice of $\rho$.

There does not seem to be a general answer
to this question, although much is known \cite{Grimmett}.
For instance, there always exists a critical probability 
$p_c$ such that 
\begin{equation}
	\P_p\left\{ \rho\leftrightarrow\infty \right\}=
	\begin{cases}
		\text{positive},&\text{if $p>p_c$},\\
		\text{zero},&\text{if $p<p_c$}.
	\end{cases}
\end{equation}
However, $\P_{p_c}\{\rho\leftrightarrow\infty\}$
can be zero for some graphs $G$, and positive for others.

When $G$ is a tree, much more is known.
In this case, \ocite{Lyons:92} has proved that
$\P_p\{\rho\leftrightarrow\infty\}>0$
if and only if there exists a probability measure
$\mu$ on $\partial G$ such that 
\begin{equation}\label{eq:Lyons}
	\iint \frac{\mu(dv)\,\mu(dw)}{p^{|v\wedge w|}}
	<\infty.
\end{equation}
In the language of population genetics, 
$v\wedge w$ denotes the ``greatest common ancestor''
of $v$ and $w$, and $|z|$ is the ``age,'' or height, of the
vertex $z$. Also, $\partial G$ denotes the boundary of
$G$. This is the collection of all infinite rays, and is
metrized with the hyperbolic metric $d(v\,,w):=
\exp(-|v\wedge w|)$. It is not hard to check that
$(\partial G\,,d)$ is a compact metric space. Furthermore,
one can apply the celebrated theorem of \ocite{Frostman}
in conjunction with
Lyons's theorem to find that
$p_c=\exp(-\dimh \partial G)$, where $\dimh\partial G$
denotes the Hausdorff dimension of the metric
space $(\partial G\,,d)$.

Lyons's theorem improves on the earlier efforts of
Lyons \citelist{\ycite{Lyons:89}\ycite{Lyons:90}} and 
aspects of the work of \ocite{DubinsFreedman}
and \ocite{Evans}. \ocite{BPP}
and \ocite{Marchal} contain two different optimal
improvements on Lyons's theorem.

H\"aggstr\"om, Peres, and Steif \ycite{HaggstromEtAl} added
equilibrium dynamics to percolation problems. Next is a 
brief description.
At time zero we construct all edge assignments
according to $\P_p$. Then we update each edge weight, independently
of all others, in a stationary-Markov fashion: If an edge is 
closed then it flips to an open one at rate $p$; if an edge weight is
open then it flips to a closed edge at rate $q:=1-p$.

Let us write
$\{\rho\stackrel{t}{\leftrightarrow}\infty\}$ for the event
that we have percolation at time $t$. By stationarity,
$\P_p\{\rho\stackrel{t}{\leftrightarrow}\infty\}$ does not
depend on $t$. In particular, if
$p<p_c$ then $\P_p\{\rho\stackrel{t}{\leftrightarrow}\infty\}=0$
for all $t\ge 0$. 
If $G$ is a tree,
then $\P_p\{\rho\stackrel{t}{\leftrightarrow}\infty\}$ 
is the probability of percolation in 
the context of \ocite{Lyons:92}. The results of
\ocite{HaggstromEtAl} imply that there exists a tree $G$
such that
$\P_{p_c} (\cup_{t>0}\{\rho\stackrel{t}{\leftrightarrow}\infty\})=1$
although $\P_{p_c}\{\rho\stackrel{t}{\leftrightarrow}\infty\}=0$ for
all $t\ge 0$.
We add that, in all cases, 
the event $\cup_{t>0}\{\rho\stackrel{t}{\leftrightarrow}\infty\}$ 
is indeed measurable, and 
thanks to ergodicity has probability
zero or one.

Now let us specialize to the case that $G$ is a
\emph{spherically symmetric} tree. This means that
all vertices of a given height have the same
number of children. In this case,
\ocite{HaggstromEtAl} studied dynamical percolation on
$G$ in greater depth and proved that
for all $p\in(0\,,1)$,
\begin{equation}\label{eq:HPS}
	\P_p\left(\bigcup_{t\ge 0}
	\left\{ \rho\stackrel{t}{\leftrightarrow}\infty\right\}
	\right) =1\quad\text{if and only if}\quad
	\sum_{ l =1}^\infty \frac{p^{-l}}{ l |G_ l |}<\infty.
\end{equation}
Here, $G_n$ denotes the collection
of all vertices of height $n$, and $|G_n|$ denotes
its cardinality.
This theorem has been extended further by \ocite{PeresSteif}.
In order to describe their results we follow
their lead and consider only the non-trivial case
where $G$ is an infinite tree. In that case,
Theorem 1.4 of \ocite{PeresSteif} asserts that for all
nonrandom closed sets 
$D\subseteq[0\,,1]$,
$\P_p( \cup_{t\in D} \{
\rho\stackrel{t}{\leftrightarrow}\infty\}
)>0$ if and only if there exists
a probability measure $\nu$ on $D$ such that
\begin{equation}\label{eq:iint}
	\iint \sum_{ l =1}^\infty \frac{1}{|G_l|}
	\left( 1+ \frac qp e^{-|t-s|} \right)^l  
	\,\nu(ds)\,\nu(dt)<\infty.
\end{equation}

The principle aim of this paper is to study
general trees $G$---i.e.,\ not necessarily spherically
symmetric ones---and describe when
there is positive probability of percolation 
for some time $t$ in a given ``target set'' $D$. Our description 
(Theorem \ref{th:main})
answers a question of Yuval Peres
(personal communication), and confirms Conjecture 1
of \ocite{PemantlePeres} for a large family of concrete
target percolations. In addition,
when $D$ is a singleton,
our description recovers the  characterization \eqref{eq:Lyons}---due
to \ocite{Lyons:92}---of
percolation on general trees.

As was mentioned earlier, it can happen that 
$\P_{p_c}\{\rho\leftrightarrow\infty\}=0$,
and yet percolation occurs at some
time $t$  [$\P_{p_c}$]. Let $S(G)$ denote the collection of
all such exceptional times.
When $G$ is spherically symmetric,
\ocite{HaggstromEtAl}*{Theorem 1.6} compute the
Hausdorff dimension
of $S(G)$. Here we do the same in the case
that $G$ is a generic tree (Theorem \ref{th:dim}).
In order to do this we appeal to the theory of 
L\'evy processes
\cites{Bertoin,Khoshnevisan,Sato}; the resulting formula 
for dimension is more complicated when $G$ is
not spherically symmetric.
We apply our formula to present simple bounds for
the Hausdorff dimension of $S(G)\cap D$ for a
non-random target set $D$ in the case that $G$ is spherically symmetric
(Proposition \ref{pr:dimh:SS}). When $D$ is a regular fractal
our upper and lower bounds agree, and we obtain an 
almost-sure identity
for the Hausdorff dimension of $S(G)\cap D$.\\

\noindent\textbf{Acknowledgements.}
I am grateful to Robin Pemantle and Yuval Peres
who introduced me to probability and analysis on
trees. Special thanks are due to Yuval Peres who
suggested the main problem that is considered here,
and to David Levin for pointing out some
typographical errors. Last but not the least, I
am grateful to an anonymous referee who
read the manuscript carefully, and made a
number of excellent suggestions as well as corrections.

\section{\bf Main Results}

Our work is in terms of various capacities for which we
need some notation.

Let $S$ be a topological space, and
suppose
$f:S\times S\to\R_+\cup\{\infty\}$ is measurable
and $\mu$ is a Borel probability measure on $S$.
Then we define the $f$-energy of $\mu$ to be
\begin{equation}
	I_f(\mu) := \iint f(x\,,y)\, \mu(dx)\, \mu(dy).
\end{equation}
We define also the $f$-capacity of a Borel set $F\subseteq S$
as
\begin{equation}
	\mathrm{Cap}_f(F) := \left[ \inf_{\mu\in\mathcal{P}(F)}
	I_f(\mu)\right]^{-1},
\end{equation}
where $\inf\varnothing:=\infty$, $1/\infty:=0$, 
and $\mathcal{P}(F)$ denotes the collection of all
probability measures on $F$. Now we return to 
the problem at hand.

For $v,w\in\partial G$ and $s,t\ge 0$ define
\begin{equation}
	h\left( (v\,,s)\,; (w\,,t)\right) :=
	\left( 1+ \frac qp e^{-|s-t|}\right)^{|v\wedge w|}.
\end{equation}
\ocite{PeresSteif} have proved that
if $\P_p\{ \rho\leftrightarrow\infty\}=0$, then
for all closed sets $D\subset[0\,,1]$,
\begin{equation}\label{eq:LB}
	\P_p\left( \bigcup_{t\in D} \left\{
	\rho\stackrel{t}{\leftrightarrow}\infty \right\}
	\right) \ge \frac12\, \mathrm{Cap}_h
	\left(\partial G\times D\right).
\end{equation}
In addition, they prove that when $G$ is spherically
symmetric,
\begin{equation}\label{eq:UB}
	\P_p\left( \bigcup_{t\in D} \left\{
	\rho\stackrel{t}{\leftrightarrow}\infty \right\}
	\right) \le 960e^3 \mathrm{Cap}_h
	\left(\partial G\times D\right).
\end{equation}
Their method is based on the fact that
when $G$ is spherically symmetric one
can identify the form of
the minimizing measure in the definition
of $\mathrm{Cap}_h(\partial G\times D)$. 
In fact, the minimizing measure can be
written as the uniform measure on $\partial
G$---see \eqref{eq:unif:G}---times some probability measure on $D$.
Whence follows also \eqref{eq:iint}.

In general, $G$ is not spherically symmetric,
thus one does not know the form of the minimizing
measure. We use other arguments that are based on
random-field methods in order to obtain 
the following result. We note that the essence
of our next theorem is in its upper bound because
it holds without any exogenous conditions.

\begin{theorem}\label{th:main}
	Suppose $\P_p\{\rho\leftrightarrow\infty\}=0$. Then,
	for all compact sets $D\subseteq\R_+$,
	\begin{equation}
		\frac12\, \mathrm{Cap}_h
		\left(\partial G \times D\right) \le
		\P_p\left( \bigcup_{t\in D} \left\{
		\rho\stackrel{t}{\leftrightarrow}\infty \right\}
		\right) \le 512\, \mathrm{Cap}_h
		\left(\partial G \times D\right).
	\end{equation}
	The condition that $\P_p\{\rho\leftrightarrow\infty\}=0$
	is not needed in the upper bound.
\end{theorem}

Thus, we can use the preceding theorem in conjunction
with \eqref{eq:LB} to deduce the following.

\begin{corollary}\label{co:main}
	Percolation occurs
	at some time $t\in D$ if and only if 
	$\partial G\times D$ has positive $h$-capacity.
\end{corollary}

We make three additional remarks.

\begin{remark}
	When $D=\{t\}$ is a singleton,
	$\mu\in\mathcal{P}(\partial G \times D)$
	if and only if $\mu(A\times B)=\nu(A)\delta_t(B)$
	for some $\nu\in\mathcal{P}(\partial G)$; also, 
	$I_h(\nu\times\delta_t) 
	=\iint p^{-|v\wedge w|}\,\nu(dv)\,\nu(dw)$.
	Therefore, Theorem \ref{th:main} contains Lyons's
	theorem \ycite{Lyons:92}, although our multiplicative constant 
	is worse than that of Lyons.
\end{remark}

\begin{remark}
	It is not too hard to modify our methods and prove that when
	$G$ is spherically symmetric,
	\begin{equation}
		\P_p\left( \bigcup_{t\in D} \left\{
		\rho\stackrel{t}{\leftrightarrow}\infty \right\}
		\right) \le \frac{512}{I_h(
		\mathfrak{m}_{_{\partial G}}\times\nu)},
	\end{equation}
	where $\mathfrak{m}_{_{\partial G}}$
	is the uniform measure on $\partial G$.
	See Theorem \ref{th:YuvalConj} below.
	From this we readily recover \eqref{eq:UB}
	with the better constant $512$ in place of
	$960e^3\approx 19282.1$. This verifies the conjecture
	of Yuval Peres that $960e^3$ is improvable
	(personal communication), although it is unlikely that
	our $512$ is optimal.
\end{remark}

\begin{remark}
	The abstract capacity condition of
	Corollary \ref{co:main} can be simplified,
	for example when $D$ is a ``strong $\beta$-set.''
	[Strong $\beta$-sets are a little more regular than
	the $s$-sets of \ocite{Besicovitch}.]
	We mention here only the following consequence of
	Theorem \ref{th:beta-set} and Example
	\ref{example:Cantor}: When $D$ is a strong
	$\beta$-set, $\P_p(\cup_{t\in D}\{\rho
	\stackrel{t}{\leftrightarrow}\infty\})>0$ if and only if
	there exists $\mu\in\mathcal{P}(\partial G)$ such that
	\begin{equation}
		\iint \frac{\mu(dv)\, \mu(dw)}{|v\wedge w|^\beta\ p^{|v\wedge w|}}
		<\infty.
	\end{equation}
	This implies both Lyons's theorem \eqref{eq:Lyons},
	and that of H\"aggstr\"om et al.\ \eqref{eq:HPS}. See remark
	\ref{rem:interpolate}.
\end{remark}

Next we follow the development of \ocite{HaggstromEtAl}*{Theorem 1.6},
and consider the Hausdorff
dimension of the set of times at which percolation
occurs. The matter is non-trivial only when $p=p_c$.

Consider the random subset $S(G):=S(G)(\omega)$ of $[0\,,1]$ defined
as
\begin{equation}
	S(G) := \left\{ t\ge 0:\ \rho\stackrel{t}{\leftrightarrow}
	\infty\right\}.
\end{equation}
Note in particular that, as events,
\begin{equation}
	\{ S(G)\cap D\neq\varnothing\}
	=\bigcup_{t\in D} \left\{ \rho\stackrel{t}{\leftrightarrow}
	\infty\right\}.
\end{equation}

Define, for all $\alpha\in(0\,,1)$, the function
$\phi(\alpha):(\partial G\times\R_+)^2 \to\R_+\cup\{\infty\}$,
as follows:
\begin{equation}\label{eq:phi}
	\phi(\alpha) \Big((v\,,t)\,;(w\,,s)\Big) := \frac{
	h((v\,,t)\,;(w\,,s))}{|t-s|^\alpha}.
\end{equation}
Then we offer the following result on the fine structure of $S(G)$.

\begin{theorem}\label{th:dim}
	Let $D$ be a non-random compact subset of $[0\,,1]$.
	If $\P_p\{\rho\leftrightarrow\infty\}>0$ then $S(G)$ has positive
	Lebesgue measure a.s.
	If $\P_p\{\rho\leftrightarrow\infty\}=0$, then $S(G)$ has
	zero Lebesgue measure a.s., and a.s.\ on 
	$\{S(G)\cap D\neq\varnothing\}$,
	\begin{equation}\begin{split}
		\dimh \left( S(G)\cap D\right)
		=  \sup\left\{ 0<\alpha<1:\
		\mathrm{Cap}_{\phi(\alpha)}(\partial G \times D)>0
		\right\},
	\end{split}\end{equation}
	where $\sup\varnothing:=0$.
\end{theorem}
When $G$ is spherically symmetric
this theorem can be simplified considerably; see 
Proposition \ref{pr:dimh:SS} below. In the case
that $G$ is spherically symmetric and $D=[0\,,1]$,
Theorem \ref{th:dim} is a consequence of Theorem 1.5 of
\ocite{PeresSteif}.

\section{\bf Proof of Theorem \ref{th:main}}
We prove only the upper bound; the lower bound \eqref{eq:LB}
was proved much earlier in \ocite{PeresSteif}. 
 
Without loss of generality we may assume that
$G$ has no leaves. Otherwise we can replace
$G$ everywhere by $G'$, where the latter
is the maximal subtree of $G$ that has no leaves.
This ``leaflessness'' assumption is in force throughout
the proof. 
Also, without loss of generality, we may assume that
$\P_p(\cup_{t\in D}\{\rho\stackrel{t}{\leftrightarrow}
\infty\} ) >0$, for there is nothing left
to prove otherwise.
 
As in \ocite{PeresSteif}, we first derive the theorem in
the case that $G$ is a finite tree. Let $n$ denote its
\emph{height}. That is, $n:=\max |v|$ where the maximum
is taken over all vertices $v$. We can---and will---assume
without loss of generality that $G$ has \emph{no leaves}.
That is, whenever a vertex $v$
satisfies $|v|<n$, then $v$ necessarily has a descendant
$w$ with $|w|=n$.

Define
\begin{equation}
	\Xi := \partial G\times D.
\end{equation}
Let $\mu\in\mathcal{P}(\Xi)$, and define
\begin{equation}\label{eq:Z}
	Z(\mu) := \frac{1}{p^n} \int_{(v,t)\in\Xi} \1_{\{
	\rho\stackrel{t}{\leftrightarrow} v\}}
	\,\mu(dv\, dt).
\end{equation}
During the course of their derivation of \eqref{eq:LB},
\ocite{PeresSteif} have demonstrated  that
\begin{equation}\label{eq:moments}
	\E_p[Z(\mu)]=1\quad\text{and}\quad
	\E_p\left[Z^2(\mu)\right] \le 2I_h(\mu).
\end{equation}
Equation \eqref{eq:LB} follows immediately from this
and the Paley--Zygmund inequality \ycite{PaleyZygmund}:
\emph{For all non-negative $f\in L^2(\P_p)$,}
\begin{equation}\label{eq:PZ}
	\P_p\{ f >0\} \ge \frac{ \left( \E_p f \right)^2}{
	\E_p[f^2]}.
\end{equation}

Now we prove the second half of Theorem \ref{th:main}.
Because $G$ is assumed to be finite, we can embed it
in the plane. We continue to write $G$ for the said embedding
of $G$ in $\R^2$; this should not cause too much confusion
since we will not refer to the abstract tree $G$ until the end
of the proof.

Since $G$ is assumed to be leafless,
we can identify $\partial G$ with the collection
of vertices $\{v:\, |v|=n\}$ of maximal length. [Recall
that $n$ denotes the height of $G$.] 

There are four natural partial orders on $\partial G\times\R_+$
which we describe next. Let $(v\,,t)$ and $(w\,,s)$ be two elements
of $\partial G\times\R_+$:
\begin{enumerate}
	\item We say that $(v\,,t)<_{_{(-,-)}} (w\,,s)$
		if $t\le s$ and $v$ lies to the left of $w$ in the planar embedding
		of $G$.
	\item If $t\ge s$ and $v$ lies to the left of $w$ [in the planar
		embedding of $G$], then we say that
		$(v\,,t)<_{_{(-,+)}}(w\,,s)$.
	\item If $t\le s$ and $v$ lies to the right of $w$, then we say that
		$(v\,,t)<_{_{(+,-)}} (w\,,s)$.
	\item If $t\ge s$ and $v$ lies to the right of $w$, then we
		say that $(v\,,t)<_{_{(+,+)}}(w\,,s)$.
\end{enumerate}
One really only needs two of these, but having four
simplifies the ensuing presentations slightly. 

The key feature
of these partial orders is that, together, they totally order
$\partial G\times\R_+$. By this we mean that
\begin{equation}\label{eq:TO}
	(v\,,t),(w\,,s)\in\partial G\times\R_+
	\ \Rightarrow\ {}^\exists \sigma,\tau \in\{-\,,+\}:\
	(v\,,t) <_{_{(\sigma,\tau)}} (w\,,s).
\end{equation}

Define, for all $(v\,,t)\in\partial G\times\R_+$ and
$\sigma,\tau\in\{-\,,+\}$,
\begin{equation}
	\mathcal{F}_{_{(\sigma,\tau)}} (v\,,t) := \text{sigma-algebra
	generated by } \left\{ \1_{\{
	\rho\stackrel{s}{\leftrightarrow}w
	\}}\,;\, (w\,,s) <_{_{(\sigma,\tau)}} (v\,,t) \right\},
\end{equation}
where the conditions
that $s\ge 0$ and $w\in\partial G$ are
implied tacitly. It is manifestly
true that for every fixed $\sigma,\tau\in\{-\,,+\}$,
the collection of sigma-algebras
$\mathcal{F}_{_{(\sigma,\tau)}}:=\{\mathcal{F}_{_{(\sigma,\tau)}}
(v\,,t)\}_{t\ge 0,
v\in\partial G}$ is a filtration in the sense that
\begin{equation}
	(w\,,s)<_{_{(\sigma,\tau)}}
	(v\,,t) \ \Longrightarrow\
	\mathcal{F}_{_{(\sigma,\tau)}}
	(w\,,s)\subseteq \mathcal{F}_{_{(\sigma,\tau)}}(v\,,t).
\end{equation}
Also, it follows fairly readily that each $\mathcal{F}_{_{(\sigma,\tau)}}$
is commuting in the sense of
\ocite{Khoshnevisan}*{pp.\ 35 and 233}. When
$(\sigma\,,\tau)=(\pm\,,+)$ this assertion is easy enough
to check directly; when $(\sigma\,,\tau)=(\pm\,,-)$, it follows
from the time-reversability of our dynamics together with the
case $\tau=+$.
Without 
causing too much confusion we can replace
$\mathcal{F}_{_{(\sigma,\tau)}}(v\,,t)$ by its completion
[$\P_p$]. Also,
we may---and will---replace the latter further by
making it right-continuous in the partial order 
$<_{(\sigma,\tau)}$. 
As a consequence of this and Cairoli's maximal inequality
\cite{Khoshnevisan}*{Theorem 2.3.2, p.\ 235}, 
for all twice-integrable random variables
$Y$, and all $\sigma,\tau\in\{-\,,+\}$,
\begin{equation}\label{eq:Cairoli1}
	\E_p\left(
	\sup_{(v,t)\in\Xi}\left| 
	\E_p \left[ Y \,\left|\, \mathcal{F}_{_{(\sigma,\tau)}}
	(v\,,t)
	\right.\right] \right|^2 \right) \le 16 \E_p\left[
	Y^2\right].
\end{equation}
In order to obtain this from Theorem 2.3.2 of Khoshnevisan
(\emph{loc.\ cit.}) set $N=p=2$, identify the parameter 
$t$ in that book by our $(v\,,t)$, and define
the process $M$ there---at time-point $(v\,,t)$---to be 
our $\E_p[Y\,|\, \mathcal{F}_{_{\sigma,\tau}}(v\,,t)]$.

Next we bound from below $\E_p[Z(\mu)\,|\,\mathcal{F}_{_{(\sigma,\tau)}}
(w\,,s)]$,
where $s\ge 0$ and $w\in\partial G$ are fixed:
\begin{equation}\begin{split}
	&\E_p\left[ Z(\mu)\,\left|\, \mathcal{F}_{_{(\sigma,\tau)}}
		(w\,,s) \right.\right]\\
	& \hskip1in\ge  \int_{
		\begin{subarray}{l}
			(v,t)\in\Xi:\\
			(w,s)<_{(\sigma,\tau)}(v,t)
		\end{subarray}}
		p^{-n}\
		\P_p\left( \left. \rho \stackrel{t}{\leftrightarrow}
		v\,\right|\, \mathcal{F}_{_{(\sigma,\tau)}}(w\,,s)\right)\,
		\mu(dv\, dt) \cdot \1_{\{ \rho\stackrel{s}{\leftrightarrow}
		w \}}.
\end{split}\end{equation}
By the Markov property, $\P_p$-a.s.\ on
$\{\rho\stackrel{s}{\leftrightarrow} w \}$,
\begin{equation}\label{eq:condcond}
	\P_p\left( \left. \rho \stackrel{t}{\leftrightarrow}
	v\,\right|\, \mathcal{F}_{_{(\sigma,\tau)}}(w\,,s)\right)
	= p^{n-|v\wedge w|}
	\left( p+q e^{-|t-s|}\right)^{|v\wedge w|}.
\end{equation}
See equation (6) of H\"aggstr\"om, Peres,
and Steif \ycite{HaggstromEtAl}. It follows then that
$\P_p$ a.s.,
\begin{equation}\label{eq:key}
	\E_p\left[ Z(\mu)\,\left|\, \mathcal{F}_{_{(\sigma,\tau)}}
	(w\,,s) \right.\right]
	\ge \int_{
	\begin{subarray}{l}
		(v,t)\in\Xi:\\
		(w,s)<_{(\sigma,\tau)}(v,t)
	\end{subarray}}
	h\left( (w\,,s)
	\,; (v\,,t)\right) \,
	\mu(dv\, dt)  \cdot \1_{\{ \rho\stackrel{s}{\leftrightarrow}
	w \}}.
\end{equation}

Thanks to the preceding, and \eqref{eq:TO},
for all $s\ge 0$ and $w\in\partial G$
the following holds $\P_p$ a.s.:
\begin{equation}\label{eq:key2}
	\sum_{\sigma,\tau\in\{-,+\}}
	\E_p\left[ Z(\mu)\,\left|\, \mathcal{F}_{_{(
	\sigma,\tau)}} (w\,,s) \right.\right]
	\ge \int_\Xi h\left( (w\,,s)
	\,; (v\,,t)\right) \,
	\mu(dv\, dt) \cdot \1_{\{ \rho\stackrel{s}{\leftrightarrow}
	w \}}.
\end{equation}
It is possible to check that the right-hand side
is a right-continuous function of $s$. Because $\partial G$
is finite, we can therefore combine all null sets and deduce that
$\P_p$ almost surely, \eqref{eq:key2} holds simultaneously
for all $s\ge 0$ and $w\in\partial G$.

Recall that we assumed, at the onset of the proof,
that $\P_p(\cup_{t\in D}
\{\rho\stackrel{t}{\leftrightarrow}\infty\})>0$. From this it follows easily that
we can find random variables $\mathbf{s}$ and $\mathbf{w}$ such that:
\begin{enumerate}
	\item $\mathbf{s}(\omega)\in D\cup\{\infty\}$ for all $\omega$,
		where $\infty$ is a point not in $\R_+$;
	\item $\mathbf{w}(\omega)\in\partial G\cup
		\{\delta\}$, where $\delta$ is an abstract
		``cemetery'' point not in $\partial G$;
	\item $(\mathbf{w}(\omega)\,,\mathbf{s}(\omega))\neq (\delta\,,\infty)$
		if and only if there exists $t\in D$ and $v\in\partial G$
		such that $\rho\stackrel{t}{\leftrightarrow}v$;
	\item $(\mathbf{w}(\omega)\,,\mathbf{s}(\omega))\neq (\delta\,,\infty)$
		if and only if $\rho\stackrel{\mathbf{s}(\omega)}{\leftrightarrow}
		\mathbf{w}(\omega)$.
\end{enumerate}
Here, and henceforth, 
$\{\rho\stackrel{t}{\leftrightarrow}v\}$ denotes the event
that at time $t$ every edge that conjoins $\rho$ and $v$ is open.

There are many ways of constructing $\mathbf{s}$ and $\mathbf{w}$.
We explain one for the sake of completeness.
Define
\begin{equation}
	\mathbf{s} := \inf\left\{ s\in D:\
	{}^\exists v\in\partial G\text{ such that }
	\rho\stackrel{s}{\leftrightarrow}v\right\},
\end{equation}
where $\inf\varnothing:=\infty$. If $\mathbf{s}(\omega)=\infty$
then we set $\mathbf{w}(\omega):=\delta$. Else, we define
$\mathbf{w}(\omega)$ to be the left-most (say) ray in 
$\partial G$ whose edges are all open at time $\mathbf{s}(\omega)$.
It might help to recall that $G$ is assumed to be a finite
tree, and we are identifying it with its planar embedding
so that ``left-most'' can be interpreted unambiguously.

Define a measure $\mu$ on $\Xi$ by letting, for all
Borel sets $A\times B\subseteq\Xi$,
\begin{equation}
	\mu(A\times B) := \P_p\left( \left. (\mathbf{w} \,,\mathbf{s})
	\in A\times B\,
	\right|\, (\mathbf{w}\,,\mathbf{s})\neq(\delta\,,\infty) \right).
\end{equation}
Note that $\mu\in\mathcal{P}(\Xi)$ because
$\P_p(\cup_{t\in D}\{\rho\stackrel{t}{\leftrightarrow}\infty\})
=\P_p\{ (\mathbf{w}\,,\mathbf{s})\neq(\delta\,,\infty)\}>0$.

We apply \eqref{eq:key2} with this particular $\mu\in\mathcal{P}(\Xi)$,
and replace $(w\,,s)$ by $(\mathbf{w}\,,\mathbf{s})$, to find that a.s.,
\begin{equation}\label{eq:key3}\begin{split}
	&\sum_{\sigma,\tau\in\{-,+\}}\sup_{(w,s)\in\Xi}
		\E_p\left[ Z(\mu)\,\left|\, \mathcal{F}_{_{(
		\sigma,\tau)}} (w\,,s) \right.\right]\\
	& \hskip2.4in \ge \int_\Xi h\left( (\mathbf{w}\,,\mathbf{s})
		\,; (v\,,t)\right) \,
		\mu(dv\, dt)  \cdot \1_{\cup_{t\in D}\{\rho
		\stackrel{t}{\leftrightarrow} \infty\}}.
\end{split}\end{equation}
According to \eqref{eq:Cairoli1},
and thanks to the inequality,
\begin{equation}\label{eq:CS}
	(a+b+c+d)^2\le 4(a^2+b^2+c^2+d^2),
\end{equation}
we can deduce that
\begin{equation}\label{eq:key4}\begin{split}
	&\E_p\left[\left(\sum_{\sigma,\tau\in\{-,+\}}\sup_{(w,s)\in\Xi}
		\E_p\left[ Z(\mu)\,\left|\, \mathcal{F}_{_{(
		\sigma,\tau)}} (w\,,s) \right.\right]
		\right)^2\right]\\
	& \hskip2.2in\le 4\sum_{\sigma,\tau\in\{-,+\}}
		\E_p\left[
		\sup_{(w,s)\in\Xi} \left|
		\E_p\left[ Z(\mu)\,\left|\, \mathcal{F}_{_{(
		\sigma,\tau)}} (w\,,s) \right.\right]
		\right|^2\right] \\
	& \hskip2.2in\le 256 \E_p\left[ Z^2(\mu)\right]\\
	& \hskip2.2in\le 512 I_h(\mu).
\end{split}\end{equation}
See \eqref{eq:moments} for the final inequality.
On the other hand, thanks to the definition of $\mu$,
and by  the Cauchy--Schwarz inequality,
\begin{equation}\label{eq:key5}\begin{split}
	&\E_p\left[\left. \left(
		\int_\Xi h\left( (\mathbf{w}\,,\mathbf{s})
		\,; (v\,,t)\right) \,
		\mu(dv\, dt)\right)^2 \,\right|\, \bigcup_{t\in D}
		\left\{ \rho \stackrel{t}{\leftrightarrow}\infty
		\right\}\right]\\
	& \hskip1.8in\ge  \left(\E_p\left[\left.
		\int_\Xi h\left( (\mathbf{w}\,,\mathbf{s})
		\,; (v\,,t)\right) \,
		\mu(dv\, dt)\,\right|\,\bigcup_{t\in D}
		\left\{ \rho \stackrel{t}{\leftrightarrow}\infty
		\right\}\right]\right)^2\\
	& \hskip1.8in = \left[ I_h(\mu)\right]^2.
\end{split}\end{equation}
Because $G$ is finite, it follows that 
$0<I_h(\mu)<\infty$.
Therefore, \eqref{eq:key3}, \eqref{eq:key4}, and \eqref{eq:key5}
together imply the theorem in the case that $G$ is finite.
The general case follows from the preceding by monotonicity.

\section{\bf Proof of Theorem \ref{th:dim}}
The assertions about the Lebesgue measure of
$S(G)$ are mere consequences of the fact
that $\P_p\{\rho\stackrel{t}{\leftrightarrow}\infty\}$
does not depend on $t$, used in conjunction with the
Fubini--Tonelli theorem. Indeed,
\begin{equation}
	\E_p\left[ \text{meas } S(G) \right]=
	\int_0^\infty \P_p\left\{ \rho\stackrel{t}{\leftrightarrow}
	\infty\right\}\, dt.
\end{equation}
Next we proceed with the remainder
of the proof.

Choose and fix $\alpha\in(0\,,1)$.
Let $Y_\alpha := \{Y_\alpha (t)\}_{t\ge 0}$ to be
a symmetric stable L\'evy process on $\R$ with
index $(1-\alpha)$. We can normalize $Y_\alpha$
so that $\E[\exp\{i\xi Y(1)\}]=\exp(-|\xi|^{1-\alpha})$
for all $\xi\in\R$.
We assume also that $Y_\alpha$ is
independent of our dynamical
percolation process. For more information on the process
$Y_\alpha$ see the monographs of \ocite{Bertoin},
\ocite{Khoshnevisan}, and \ocite{Sato}. 

Recall the function $\phi(\alpha)$ from \eqref{eq:phi}.
Our immediate goal
is to demonstrate the following.

\begin{theorem}\label{th:hit:Y}
	Suppose $\P_p\{\rho\leftrightarrow\infty\}=0$.
	Choose and fix $M>1$. Then there exists
	a finite constant $A=A(M)>1$ such that
	for all compact sets $D\subseteq [-M\,,M]$,
	\begin{equation}
		\frac 1A \mathrm{Cap}_{\phi(\alpha)} (\partial G\times D) \le 
		\P_p\left\{ S(G)\cap D\cap \overline{Y_\alpha([1\,,2])}
		\neq\varnothing\right\} \le A \mathrm{Cap}_{\phi(\alpha)}
		(\partial G\times D),
	\end{equation}
	where $\overline{\rm U}$ denotes the closure of $\rm U$.
	The condition that $\P_p\{\rho\leftrightarrow\infty\}=0$
	is not needed in the upper bound.
\end{theorem}

\begin{remark}
	The time interval $[1\,,2]$ could just as easily
	be replaced with an arbitrary, but fixed, closed interval 
	$I\subset(0\,,\infty)$ that is
	bounded away from the origin.
\end{remark}

\begin{remark}\label{rem:hit:Y}
	We do not require the following strengthened
	form of Theorem \ref{th:hit:Y}, but it is
	simple enough to derive that we 
	describe it here: \emph{Theorem \ref{th:hit:Y}
	continues to holds if
	$\overline{Y_\alpha([1\,,2])}$ 
	were replaced by $Y_\alpha([1\,,2])$.}
	Indeed, a well-known theorem of \ocite{Kanda}
	implies that all semipolar sets for $Y_\alpha$
	are polar; i.e., $Y_\alpha$ satisfies Hunt's
	(H) hypothesis
	\cites{Hunt1,Hunt3}. 
	This readily implies  the assertion of this Remark. 
\end{remark}

From here on, until after the completion of
the proof of Theorem \ref{th:hit:Y}, 
we will assume without loss of much generality that 
$G$ is a finite tree of height $n$. The extension
to the case where $G$ is infinite is made by 
standard arguments.

Let $D$ be as in Theorem \ref{th:hit:Y}.
For all $\mu\in\mathcal{P}(\partial G\times D)$
and $\e\in(0\,,1)$ define
\begin{equation}
	Z_\e(\mu) :=  \frac{1}{(2\e)p^n}
	\int_\Xi \int_1^2 \1_{\{
	\rho\stackrel{t}{\leftrightarrow} v\}\cap
	\{ |Y_\alpha (r)-t|\le \e\}}\, dr
	\,\mu(dv\, dt),
\end{equation}
where $\Xi:=\partial G\times D$, as before.

Next
we collect some of the 
elementary properties of $Z_\e(\mu)$.

\begin{lemma}\label{lem:Zeps:moments}
	There exists $c>1$ such that
	for all $\mu\in\mathcal{P}(\Xi)$ and $\e\in(0\,,1)$:
	\begin{enumerate}
		\item $\E_p[Z_\e(\mu)]\ge 1/c$; and
		\item $\E_p[Z^2_\e(\mu)]\le c I_{\phi (\alpha)} (\mu)$.
	\end{enumerate}
\end{lemma}

\begin{proof}
	Define $p_r(a)$ to be the density of $Y_\alpha (r)$ at $a$. 
	Bochner's \emph{subordination} implies that: (i) $p_r(a)>0$ for all $r>0$
	and $a\in\R$; and (ii) there exists $c_1>0$
	such that $p_r(a) \ge c_1$
	uniformly for all $r\in[1\,,2]$ and $a\in[-M-1\,,M+1]$. We recall
	here that subordination is the assertion
	that $p_r(a)$ is a stable mixture
	of Gaussian densities \cite{Bochner}. 
	For related results, see also \ocite{Nelson}.
	\ocite{Khoshnevisan}*{pp.\ 377--384} contains a modern-day account
	that includes explicit proofs of assertions (i) and (ii) above. 
	The first assertion of the lemma follows.
	
	Next we can note that by
	the Markov property of $Y_\alpha$,
	\begin{equation}\begin{split}
		&\int_1^2\int_1^2 \P\left\{
			| Y_\alpha(r)-t|\le\e \,,
			| Y_\alpha(R)-s|\le\e\right\}\, dr\, dR\\
		&\qquad \le \int_1^2\int_r^2 \P\left\{|Y_\alpha(r)-t|\le
			\e\right\}\P\left\{ |Y_\alpha(R-r)-(s-t)|\le
			2\e\right\}\, dR\,dr\\
		& \hskip1.1in + \int_1^2\int_R^2\P\left\{|Y_\alpha(R)-s|\le
			\e\right\}\P\left\{ |Y_\alpha(r-R)-(s-t)|\le
			2\e\right\}\, dr\,dR.
	\end{split}\end{equation}
	We can appeal to subordination once again to find that there
	exists $c_2>0$ such that $p_r(a)\le c_2$ uniformly for all 
	$r\in[1\,,2]$ and $a\in[-M-1\,,M+1]$. This, and symmetry,
	together imply that
	\begin{equation}\begin{split}
		&\int_1^2\int_1^2 \P\left\{
			| Y_\alpha(r)-t|\le\e \,,
			| Y_\alpha(R)-s|\le\e\right\}\, dr\, dR\\
		& \hskip2.6in \le 4c_2\e\int_0^2 
			\P\left\{ |Y_\alpha(r)-(t-s)|\le
			2\e\right\}\, dr\\
		&\hskip2.6in \le 4e^2c_2\e \int_0^\infty 
			\P\left\{ |Y_\alpha(r)-(t-s)|\le
			2\e\right\}e^{-r}\, dr.
	\end{split}\end{equation}
	Let $u(a):=\int_0^\infty p_t(a)e^{-t}\, dt$ denote
	the one-potential density of $Y_\alpha$, and note that
	\begin{equation}
		\int_1^2\int_1^2 \P\left\{
		| Y_\alpha(r)-t|\le\e \,,
		| Y_\alpha(R)-s|\le\e\right\}\, dr\, dR
		\le 4e^2c_2\e \int_{|t-s|-2\e}^{
		|t-s|+2\e} u(z)\, dz.
	\end{equation}
	It is well known that there exist $c_3>1$ such that
	\begin{equation}\label{eq:u}
		\frac{1}{c_3|z|^\alpha}\le
		u(z) \le \frac{c_3}{|z|^\alpha},
		\qquad\text{for all $z\in[-2M-2\,,2M+2]$};
	\end{equation}
	see \ocite{Khoshnevisan}*{Lemma 3.4.1, p.\ 383},
	for instance.
	It follows that
	\begin{equation}
		\int_1^2\int_1^2 \P\left\{
		| Y_\alpha(r)-t|\le\e \,,
		| Y_\alpha(R)-s|\le\e\right\}\, dr\, dR
		\le 4e^2c_2c_3\e \int_{|t-s|-2\e}^{
		|t-s|+2\e} \frac{dz}{|z|^\alpha}.
	\end{equation}
	If $|t-s|\ge 4\e$, then we use the bound
	$|z|^{-\alpha}\le (|t-s|/2)^{-\alpha}$. Else, we use
	the estimate $\int_{|t-s|-2\e}^{|t-s|+2\e}(\cdots)
	\le\int_{-6\e}^{6\e}(\cdots)$. This leads us
	to the existence of a constant $c_4=c_4(M)>0$ such that for all $s,t\in D$
	and $\e\in(0\,,1)$, 
	\begin{equation}\label{eq:EZ22}\begin{split}
		\int_{|t-s|-2\e}^{|t-s|+2\e}
			\frac{dz}{|z|^\alpha} &\le c_4\e
			\min\left(\frac{1}{|t-s|}\wedge\frac 1\e\right)^\alpha\\
		&\le c_4\frac{\e}{|t-s|^\alpha}.
	\end{split}\end{equation}
	Part two of the lemma follows from this and \eqref{eq:condcond}.
\end{proof}

Now we prove the first inequality in Theorem \ref{th:hit:Y}.

\begin{proof}[Proof of Theorem \ref{th:hit:Y}: First Half]
	We can choose some $\mu\in\mathcal{P}(\Xi)$, and deduce
	from Lemma \ref{lem:Zeps:moments} and the Paley--Zygmund
	inequality \eqref{eq:PZ} that
	$\P_p\{ Z_\e(\mu)>0\} \ge 1/(c^3 I_{\phi(\alpha)}
	(\mu))$. Let $Y^\e$ denote the closed $\e$-enlargement
	of $Y_\alpha([1\,,2])$.
	
	Recall that $S(G)$ is closed because
	$\P_p\{\rho\leftrightarrow\infty\}=0$ \cite{HaggstromEtAl}*{Lemma 3.2}.
	Also note that
	\begin{equation}
		\{Z_\e(\mu)>0\}\subseteq\{ 
		S(G)\cap D\cap Y^\e\neq\varnothing\}. 
	\end{equation}
	Let $\e\to 0^+$
	to obtain the first inequality of Theorem
	\ref{th:hit:Y} after we
	optimize over $\mu\in\mathcal{P}(\Xi)$.
\end{proof}

The second half of Theorem \ref{th:hit:Y} is more difficult
to prove. We begin by altering the definition of $Z_\e(\mu)$
slightly as follows: For all $\e\in(0\,,1)$ and
$\mu\in\mathcal{P}(\Xi)$ define
\begin{equation}
	W_\e(\mu) :=  \frac{1}{(2\e)p^n}
	\int_\Xi \int_1^\infty \1_{\{
	\rho\stackrel{t}{\leftrightarrow} v\}\cap
	\{ |Y_\alpha (r)-t|\le \e\}} e^{-r}\, dr
	\,\mu(dv\, dt).
\end{equation}
[It might help to recall that $n$ 
denotes the height of the finite
tree $G$.] We can sharpen the second assertion
of Lemma \ref{lem:Zeps:moments}, and replace $Z_\e
(\mu)$ by $W_\e(\mu)$, as follows: There exists 
a constant $c=c(M)>0$ such that
\begin{equation}\label{EW2}
	\E_p\left[ W_\e^2(\mu) \right] \le
	c I_{\phi_\e(\alpha)}(\mu),
\end{equation}
where
\begin{equation}
	\phi_\e(\alpha) \left( (v\,,t)\,;(w\,,s)\right)
	:= h\left( (v\,,t)\,;(w\,,s)\right) \cdot
	\left( \frac{1}{|t-s|}\wedge 
	\frac1\e\right)^\alpha.
\end{equation}
The aforementioned sharpening rests on \eqref{eq:EZ22}
and not much more. So we omit the details.

Define $\mathcal{Y} (t)$ to be the sigma-algebra
generated by $\{ Y_\alpha(r)\}_{0\le r\le t}$. 
We can
add to $\mathcal{Y} (t)$ all $\P$-null sets, and even
make it right-continuous [with respect to the usual total 
order on $\R$]. Let us denote the resulting sigma-algebra
by $\mathcal{Y}(t)$ still, and the corresponding
filtration by $\mathcal{Y}$. 

Choose and fix $\sigma,\tau\in\{-\,,+\}$, and for
all $v\in\partial G$ and $r,t\ge 0$ define
\begin{equation}
	\mathcal{G}_{_{(\sigma,\tau)}} (v\,,t\,,r) :=
	\mathcal{F}_{_{(\sigma,\tau)}} (v\,,t) \times
	\mathcal{Y}(r).
\end{equation}
We say that $(v\,,t\,,r)\ll_{_{(\sigma,\tau)}}(w\,,s\,,u)$
when $(v\,,t)<_{_{(\sigma,\tau)}}(w\,,s)$ and
$r\le u$. Thus, each $\ll_{{(\sigma,\tau)}}$ defines a
partial order on $\partial G\times\R_+\times\R_+$.

Choose and fix $\sigma,\tau\in\{-\,,+\}$.
Because $\mathcal{F}_{_{(\sigma,\tau)}}$ is a two-parameter,
commuting filtration in the partial order
$<_{_{(\sigma,\tau)}}$, and since 
$\mathcal{Y}$ is the [one-parameter] independent
filtration generated by a reversible Feller process, it follows
readily that $\mathcal{G}_{_{(\sigma,\tau)}}$ is a three-parameter,
commuting filtration in the partial order
$\ll_{_{(\sigma,\tau)}}$. In particular, the following analogue of
\eqref{eq:Cairoli1} is valid: For all $V\in L^2(\P_p)$,
\begin{equation}\label{eq:Cairoli2}
	\E_p\left(
	\sup_{(v,t,r)\in\Xi\times\R_+}\left| 
	\E_p \left[ V \,\left|\, \mathcal{F}_{_{(\sigma,\tau)}}
	(v\,,t\,,r)
	\right.\right] \right|^2 \right) \le 64 \E_p\left[
	V^2\right].
\end{equation}
\cite{Khoshnevisan}*{Theorem 2.3.2, p.\ 235}.

Next, we note that or all $(w\,,s\,,u)\in\partial G\times D\times[1\,,2]$,
and all $\sigma,\tau\in\{-\,,+\}$, the following is valid
$\P_p$-almost surely:
\begin{equation}\label{eq:hit:LB}\begin{split}
	&\E_p\left[ W_\e(\mu) \,\left|\,
		\mathcal{G}_{_{(\sigma,\tau)}} (w\,,s\,,u) \right. \right]\\
	& \hskip1.in \ge \frac{1}{(2\e)p^n}
		\int_{(v,t)<_{_{(\sigma,\tau)}}
		(w,s)}\ \int_u^\infty \mathcal{H}
		e^{-r} \, dr\, \mu(dv\,dt)\cdot\1_{\{ \rho\stackrel{s}{\leftrightarrow}w
		\}\cap\{|Y_\alpha(u)-s|\le\e/2\}}.
\end{split}\end{equation}
Here,
\begin{equation}\begin{split}
	\mathcal{H} &:=
		\P_p\left( \left.\rho\stackrel{t}{\leftrightarrow}v
		\,, |Y_\alpha(r)-t|\le\e 
		\,\right|\, \mathcal{G}_{_{(\sigma,\tau)}} (w\,,s\,,u)\right)\\
	&= \P_p\left( \left.\rho\stackrel{t}{\leftrightarrow}v
		\,\right|\, \mathcal{F}_{_{(\sigma,\tau)}}(w\,,s)\right) \times
		\P\left( \left. \left| Y_\alpha(r)-t\right|\le\e\,
		\right|\, \mathcal{Y}(u)\right).
\end{split}\end{equation}
By the Markov property, $\P_p$-almost surely on
$\{\rho\stackrel{s}{\leftrightarrow}w\}$,
\begin{equation}
	\P_p\left( \left.\rho\stackrel{t}{\leftrightarrow}v
	\,\right|\, \mathcal{F}_{_{(\sigma,\tau)}}(w\,,s)\right)
	= p^n h\left( (v\,,t)\,;(w\,,s)\right).
\end{equation}
See \eqref{eq:condcond}. On the other hand, the Markov property
of $Y_\alpha$ dictates that almost surely on
$\{|Y_\alpha(u)-s|\le\e/2\}$,
\begin{equation}\begin{split}
	\P\left( \left. \left| Y_\alpha(r)-t\right|\le\e\,
		\right|\, \mathcal{Y}(u) \right) &\ge
		\P\left\{  \left| Y_\alpha(r-u)-(t-s)\right|\le
		\frac\e 2 \right\}\\
	&:= \mathcal{A}.
\end{split}\end{equation}
Note that because $u\in[1\,,2]$,
\begin{equation}\begin{split}
	\int_u^\infty \mathcal{A} e^{-r}\, dr &\ge \frac{1}{e^2}
		\int_0^\infty
		\P\left\{  \left| Y_\alpha(r)-(t-s)\right|\le
		\frac\e 2 \right\} e^{-r}\, dr\\
	& = \frac{1}{e^2}\int_{|t-s|-(\e/2)}^{|t-s|+(\e/2)}
		u(z)\, dz\\
	&\ge c_5(2\e)
		\min\left(\frac{1}{|t-s|}\wedge\frac 1\e\right)^\alpha,
\end{split}\end{equation}
where $c_5$ does not depend on $(\e\,,\mu\,;t\,,s)$.
The ultimate inequality follows from a
similar argument that was used earlier to derive \eqref{eq:EZ22}.
So we omit the details.

Thus, we can plug the preceding bounds into \eqref{eq:hit:LB}
and deduce that $\P_p$-a.s.,
\begin{equation}\label{eq:sum}\begin{split}
	&\E_p\left[ W_\e(\mu) \,\left|\,
		\mathcal{G}_{_{(\sigma,\tau)}} (w\,,s\,,u) \right. \right]\\
	&\qquad \ge c_5 \int_{(v,t)<_{_{(\sigma,\tau)}}
		(w,s)}\ \phi_\e(\alpha)\Big( 
		(v\,,t)\,;(w\,,s)\Big)
		\, \mu(dv\,dt)
		\times \1_{\{ \rho\stackrel{s}{\leftrightarrow}w
		\}\cap\{|Y_\alpha(u)-s|\le\e/2\}}.
\end{split}\end{equation}
Moreover, it is possible to check that there exists
one null set outside which the preceding holds
for all $(w\,,s\,,u)\in\Xi\times[1\,,2]$.
We are in a position to complete our proof of Theorem 
\ref{th:hit:Y}.

\begin{proof}[Proof of Theorem \ref{th:hit:Y}: Second Half]
	Without loss of very much generality, we may assume that
	\begin{equation}
		\P_p\{ S(G)\cap D\cap\overline{Y_\alpha([1\,,2])}\}>0,
	\end{equation}
	for otherwise there is nothing to prove.

	Let us introduce two abstract cemetery states: 
	$\delta\not\in\partial G$
	and $\infty\not\in\R_+$. 
	Then, there exists a map 
	$(\mathbf{w}_\e\,,\mathbf{s}_\e\,,
	\mathbf{u}_\e):\Omega\mapsto (\partial G\cup\{\delta\})
	\times (D\cup\{\infty\})\times
	([1\,,2]\cup\{\infty\})$ with the properties
	that:
	\begin{enumerate}
		\item $(\mathbf{w}_\e\,,
			\mathbf{s}_\e\,,\mathbf{u}_\e)(\omega)
			\neq (\delta\,,\infty\,,\infty)$ if and only if
			there exists $(w\,,s\,,u)(\omega)
			\in\Xi\times[1\,,2]$ such that
			$\rho\stackrel{s(\omega)}{\leftrightarrow}w(\omega)$ and
			$|Y_\alpha(u)-s|(\omega) \le\e/2$; and
		\item If $(\mathbf{w}_\e\,,\mathbf{s}_\e
			\,,\mathbf{u}_\e)(\omega)
			\neq (\delta\,,\infty\,,\infty)$, then (1) 
			holds with 
			$(\mathbf{w}_\e\,,\mathbf{s}_\e\,,
				\mathbf{u}_\e)(\omega)$
			in place of $(w\,,s\,,u)(\omega)$.
	\end{enumerate}
	See the proof of Theorem \ref{th:main} for a closely-related
	construction.
	
	Consider the event,
	\begin{equation}
		H(\e) := \left\{\omega:\
		(\mathbf{w}_\e\,,\mathbf{s}_\e\,,\mathbf{u}_\e)
		(\omega)
		\neq (\delta\,,\infty\,,\infty)\right\}.
	\end{equation}
	Thus, we can deduce that $\mu_\e\in\mathcal{P}
	(\Xi)$, where
	\begin{equation}
		\mu_\e (A\times B) := \P_p\left(\left.
		\left( \mathbf{w}_\e\,,\mathbf{s}_\e\right)\in
		A\times B
		\,\right|\, H(\e) \right),
	\end{equation}
	valid for all measurable $A\times B\subseteq\Xi$.
	
	Because of \eqref{eq:TO},
	we may apply \eqref{eq:sum} with $\mu_\e$ in place
	of $\mu$ and $(\mathbf{w}_\e\,,\mathbf{s}_\e\,,\mathbf{u}_\e)$
	in place of $(w\,,s\,,u)$ to find that $\P_p$-a.s.,
	\begin{equation}\label{eq:sumsum}\begin{split}
		&\sum_{\sigma,\tau\in\{-,+\}}
			\sup_{(w,s,u)\in\Xi\times[1,2]}
			\E_p\left[ W_\e(\mu_\e) \,\left|\,
			\mathcal{G}_{_{(\sigma,\tau)}} (w\,,s\,,u) \right. \right]\\
		& \hskip2.2in \ge c_5 \int_\Xi \phi_\e(\alpha)\Big( 
			(v\,,t)\,;(\mathbf{w}_\e\,,\mathbf{s}_\e)\Big)
			\, \mu_\e (dv\,dt)
			\cdot \1_{ H(\e)}.
	\end{split}\end{equation}
	We can square both ends of this inequality and take expectations
	[$\P_p$]. Owing to \eqref{eq:CS},
	the expectation of the square of the left-most term
	is at most
	\begin{equation}\label{eq:goal1}\begin{split}
		4 \sum_{\sigma,\tau\in\{-,+\}}
			\E_p\left(
			\sup_{(w,s,u)\in\Xi\times[1,2]}\left|
			\E_p\left[ W_\e(\mu_\e) \,\left|\,
			\mathcal{G}_{_{(\sigma,\tau)}} (w\,,s\,,u) \right. \right]
			\right|^2 \right) &\le 1024 \E_p\left[ W_\e^2(\mu_\e) \right]\\
		& \le 1024 c I_{\phi_\e(\alpha)}(\mu_\e).
	\end{split}\end{equation}
	See \eqref{eq:Cairoli2} and \eqref{EW2}. We emphasize that
	the constant $c$ is finite and positive, and does not
	depend on $(\e\,,\mu_\e)$.
	
	On the other hand, by the Cauchy--Schwarz inequality,
	the expectation of the square
	of the right-most term in \eqref{eq:sumsum} is
	equal to
	\begin{equation}\label{eq:goal2}\begin{split}
		&c_5^2 \E_p\left[\left.\left(
			\int_\Xi \phi_\e(\alpha)\Big( 
			(v\,,t)\,;(\mathbf{w}_\e\,,\mathbf{s}_\e)\Big)
			\, \mu_\e (dv\,dt)\right)^2\,\right|\,
			H(\e) \right] \P_p(H(\e))\\
		&\hskip1in\ge c_5^2 \left(\E_p\left[\left.
			\int_\Xi \phi_\e(\alpha)\Big( 
			(v\,,t)\,;(\mathbf{w}_\e\,,\mathbf{s}_\e)\Big)
			\, \mu_\e (dv\,dt)\,\right|\,
			H(\e) \right]\right)^2 \P_p(H(\e))\\
		&\hskip1in = c_5^2 \left( I_{\phi_\e(\alpha)}(\mu_\e)\right)^2
			\P_p(H(\e)).
	\end{split}\end{equation}
	We can combine \eqref{eq:goal1} with \eqref{eq:goal2} to 
	find that for all $N>0$,
	\begin{equation}\begin{split}
		\P_p(H(\e)) &\le \frac{1024c}{c_5^2}
			\left[ I_{\phi_\e(\alpha)}(\mu_\e)\right]^{-1}\\
		&\le\frac{1024c}{c_5^2}
			\left[ I_{N\wedge \phi_\e(\alpha)}(\mu_\e)\right]^{-1}.
	\end{split}\end{equation}
	Perhaps it is needless to say that $N\wedge\phi_\e(\alpha)$
	is the function whose evaluation at $((v\,,t)\,,(w\,,s))\in\Xi\times\Xi$ is
	the minimum of the constant $N$ and $\phi_\e(\alpha)(
	(v\,,t)\,(w\,,s))$. Evidently, $N\wedge\phi_\e(\alpha)$ is
	bounded and lower semicontinuous on $\Xi\times\Xi$. By compactness
	we can find $\mu_0\in\mathcal{P}(\Xi)$ such that $\mu_\e$
	converges weakly to $\mu_0$. As $\e\downarrow 0$,
	the sets $H(\e)$ decrease set-theoretically, and their 
	intersection includes
	$\{S(G)\cap D\cap \overline{Y_\alpha([1\,,2])}
	\neq\varnothing\}$. As a result we have
	\begin{equation}
		\P_p\left\{ S(G)\cap D\cap \overline{Y_\alpha([1\,,2])}
		\neq\varnothing\right\} \le \frac{1024 c}{c_5^2}
		\left[ I_{N\wedge \phi(\alpha)}(\mu_0)\right]^{-1}.
	\end{equation}
	Let $N\uparrow\infty$ and appeal to the monotone convergence theorem
	to find that
	\begin{equation}\begin{split}
		\P_p\left\{ S(G)\cap D\cap \overline{Y_\alpha([1\,,2])}
			\neq\varnothing\right\} &\le \frac{1024 c}{c_5^2}
			\left[ I_{\phi(\alpha)}(\mu_0)\right]^{-1}\\
		&\le \frac{1024 c}{c_5^2} \mathrm{Cap}_{\phi(\alpha)}(
			\partial G\times D).
	\end{split}\end{equation}
	This concludes the proof of Theorem \ref{th:hit:Y}.
\end{proof}

\begin{proof}[Proof of Theorem \ref{th:dim}]
	Define $R_\alpha := \overline{Y_\alpha([1\,,2])}$,
	and recall the theorem of \ocite{McKean}:
	\emph{For all Borel sets $B\subseteq\R$,}
	\begin{equation}
		\P\left\{ R_\alpha\cap B\neq\varnothing\right\}
		>0\quad\Longleftrightarrow\quad \mathrm{Cap}_\alpha(B)>0.
	\end{equation}
	Here, $\mathrm{Cap}_\alpha(B)$ denotes the $\alpha$-dimensional
	(Bessel-) Riesz capacity of $B$ \cite{Khoshnevisan}*{Appendix C}.
	That is,
	\begin{equation}
		\mathrm{Cap}_\alpha(B):=\left[
		\inf_{\mu\in\mathcal{P}(B)} I_\alpha(\mu)\right]^{-1},
	\end{equation}
	where
	\begin{equation}
		I_\alpha(\mu) := \iint \frac{\mu(dx)\,\mu(dy)}{|x-y|^\alpha}.
	\end{equation}
	
	Now let $R^1_\alpha,R^2_\alpha,\ldots$
	be i.i.d.\ copies of $R_\alpha$, all independent
	of our dynamical percolation process as well. Then, by
	the Borel--Cantelli lemma,
	\begin{equation}\label{eq:codim1}
		\P\left\{ \bigcup_{j=1}^\infty
		R^j_\alpha\cap B\neq\varnothing\right\} =
		\begin{cases}
			1,&\text{if }\mathrm{Cap}_\alpha(B)>0,\\
			0,&\text{if }\mathrm{Cap}_\alpha(B)=0.
		\end{cases}
	\end{equation}
	Set $B:=S(G)\cap D$ and condition, once on 
	$G$ and once on $\cup_{j=1}^\infty R^j_\alpha$. 
	Then, the preceding and Theorem \ref{th:hit:Y} together imply that 
	\begin{equation}\label{eq:capcap}
		\P_p\Big\{
		\mathrm{Cap}_\alpha\left(S(G)\cap D
		\right)>0\Big\}=
		\begin{cases}
			1,&\text{if}\quad\mathrm{Cap}_{\phi(\alpha)}
				(\partial G\times D)>0,\\
			0,&\text{if}\quad\mathrm{Cap}_{\phi(\alpha)}
				(\partial G\times D)=0.
		\end{cases}
	\end{equation}
	The remainder of the theorem follows from 
	Frostman's theorem \ycite{Frostman}: {\it For all Borel sets
	$F\subset\R$,}
	\begin{equation}
		\dimh F = \sup\{0<\alpha<1:\
		\mathrm{Cap}_\alpha(F)>0\},
	\end{equation}
	which we apply
	with $F:= S(G)\cap D$. For a pedagogic
	account of Frostman's theorem
	see \ocite{Khoshnevisan}*{Theorem 2.2.1, p.\ 521}.
\end{proof}

\newpage
\section{\bf On Spherically Symmetric Trees}\label{sec:SST}

Suppose $G$ is spherically symmetric, and
let $\mathfrak{m}_{_{\partial G}}$ denote the uniform measure on $\partial G$.
One way to define $\mathfrak{m}_{_{\partial G}}$ is as follows:
For all $f:\mathbf{Z}_+\to\R_+$ and all $v\in\partial G$,
\begin{equation}\label{eq:unif:G}
	\int_{\partial G} f ( |v\wedge w|)\, \mathfrak{m}_{_{\partial G}}(dw)
	=\sum_{l=0}^{n-1} \frac{f( l )}{|G_l|},
\end{equation}
where $n\in\mathbf{Z}_+\cup\{\infty\}$ denotes the height
of $G$. In particular, we may
note that if $G$ is infinite then for all $\nu\in\mathcal{P}(\R_+)$,
\begin{equation}
	I_h(\mathfrak{m}_{_{\partial G}}\times\nu) =
	\iint \sum_{l=0}^\infty \frac{1}{|G_l |}
	\left( 1+\frac qp e^{-|t-s|}\right)^l  \, \nu(ds)\, \nu(dt).
\end{equation}
This is the integral in \eqref{eq:iint}.

Yuval Peres asked us if the constant $960e^3\approx 19282.1$
in \eqref{eq:UB}
can be improved upon. The following answers this question
by replacing $960e^3$ by $512$. Although we do not know how
to improve this constant further, it seems unlikely to be
the optimal one.

\begin{theorem}\label{th:YuvalConj}
	Suppose $G$ is an infinite, spherically symmetric tree,
	and $\P_p\{\rho\leftrightarrow\infty\}=0$. Then, for all
	compact sets $D\subseteq[0\,,1]$,
	\begin{equation}
		\frac{1}{2\inf_{\nu\in\mathcal{P}(D)}
		I_h(\mathfrak{m}_{_{\partial G}}\times\nu)}\le 
		\P_p\left( \bigcup_{t\in D} \left\{
		\rho\stackrel{t}{\leftrightarrow}\infty \right\}
		\right) \le \frac{512}{\inf_{\nu\in\mathcal{P}(D)}
		I_h(\mathfrak{m}_{_{\partial G}}\times\nu)},
	\end{equation}
	where $\inf\varnothing:=\infty$ and $1/\infty:=0$.
	The condition that $\P_p\{\rho\leftrightarrow\infty\}=0$
	is not needed for the upper bound.
\end{theorem}

The proof follows that of Theorem \ref{th:main} closely.
Therefore, we sketch the highlights of the proof only.

\begin{proof}[Sketch of Proof]
	The lower bound follows immediately from \eqref{eq:LB},
	so we concentrate on the upper bound only. As we have done before,
	we may, and will, assume without loss of generality that
	$G$ is a finite tree of height $n$.
	
	For all
	$\nu\in\mathcal{P}(D)$ consider $Z(\mathfrak{m}_{_{\partial G}}\times\nu)$
	defined in \eqref{eq:Z}. That is,
	\begin{equation}
		Z(\mathfrak{m}_{_{\partial G}}\times\nu) = \frac{1}{p^n|G_n|} \int_D \sum_{v\in G_n}
		\1_{\{\rho\stackrel{t}{\leftrightarrow}v\}} \, \nu(dt).
	\end{equation}
	It might help to point out that in the present setting,
	$G_n$ is identified with $\partial G$.
	According to \eqref{eq:moments},
	\begin{equation}\label{eq:EZ222}
		\E_p[Z(\mathfrak{m}_{_{\partial G}}\times\nu)]=1\quad\text{and}\quad
		\E_p\left[ Z^2(\mathfrak{m}_{_{\partial G}}\times\nu)\right] \le
		2I_h(\mathfrak{m}_{_{\partial G}}\times\nu).
	\end{equation}
	
	In accord with \eqref{eq:key2}, outside a single
	null set, the following holds for all
	$w\in\partial G$, $\sigma,\tau\in\{-\,,+\}$,
	and $s\ge 0$:
	\begin{equation}\begin{split}
		&\sum_{\sigma,\tau\in\{-,+\}}
			\E_p\left[ Z(\mathfrak{m}_{_{\partial G}}
			\times\nu)\,\left|\, \mathcal{F}_{_{(
			\sigma,\tau)}} (w\,,s) \right.\right]\\
		&\hskip2in\ge \int_{
			(v,t)\in\Xi} \left( 1 + \frac{q}{p} e^{-|t-s|}
			\right)^{|v\wedge w|} \,
			(\mathfrak{m}_{_{\partial G}}\times\nu)(dv\, dt)\cdot 
			\1_{\{ \rho\stackrel{s}{\leftrightarrow}
			w \}}\\
		& \hskip2in =\int_D \sum_{l=0}^{n-1} \frac{1}{|G_l |}
			\left( 1+ \frac{q}{p}
			e^{-|t-s|}\right)^l  \, \nu(dt)\times 
			\1_{\{ \rho\stackrel{s}{\leftrightarrow}
			w \}}\\
		& \hskip2in = I_h(\mathfrak{m}_{_{\partial G}}\times\nu)\cdot 
			\1_{\{ \rho\stackrel{s}{\leftrightarrow}
			w \}}.
	\end{split}\end{equation}
	See \eqref{eq:unif:G} for the penultimate line.
	Next we take the supremum of the left-most term
	over all $w$, and replace $w$ by $\mathbf{w}$
	in the right-most term; then
	square and take expectations, as we did in the course
	of the proof of Theorem \ref{th:main}.
\end{proof}

Finally, let us return briefly to the Hausdorff dimension
of $S(G)\cap D$ in the case that $G$ is spherically symmetric.
First we recall Theorem 1.6 of \ocite{HaggstromEtAl}: 
\emph{If $\P_p(\cup_{t\in D}\{\rho\stackrel{t}{\leftrightarrow}
\infty\})=1$ then $\P_p$-a.s.,}
\begin{equation}\label{eq:dimh}
	\dimh S(G) = \sup\left\{ \alpha>0:\
	\sum_{ l =1}^\infty \frac{p^{-l} l ^{\alpha-1}}{
	|G_ l |}<\infty \right\}.
\end{equation}
Next we announce the Hausdorff dimension of $S(G)\cap D$
in the case that $G$ is spherically symmetric.

\begin{theorem}\label{th:dimh:SS}
	Suppose that the tree $G$ is infinite and spherically symmetric.
	If, in addition, 
	$\P_p(\cup_{t\in D}\{\rho\stackrel{t}{\leftrightarrow}\infty\})>0$,
	then for all
	compact sets $D\subseteq[0\,,1]$,
	\begin{equation}
		\dimh \left( S(G)\cap D \right) =
		\sup\left\{ 0<\alpha<1:\ \inf_{\nu\in\mathcal{P}(D)}
		I_{\phi(\alpha)}(\mathfrak{m}_{_{\partial G}}\times\nu)<\infty
		\right\},
	\end{equation}
	$\P_p$-almost surely on $\{S(G)\cap D\neq\varnothing\}$.
\end{theorem}
The strategy of the proof is exactly the same as that of
the proof of Theorem \ref{th:dim}, but we use
$Z(\mathfrak{m}_{_{\partial G}}\times\nu)$ in place of $Z(\mu)$. The minor differences
in the proofs are omitted here. 

For the purposes of comparison,
we mention the following consequence of \eqref{eq:unif:G}:
For all $\nu\in\mathcal{P}(\R_+)$,
\begin{equation}\label{eq:I:cross}\begin{split}
	I_{\phi(\alpha)}(\mathfrak{m}_{_{\partial G}}\times\nu) &=
		\iint \sum_{l=0}^\infty \frac{1}{|G_l|}
		\left( 1+\frac qp e^{-|t-s|}\right)^l  \frac{\nu(ds)\, \nu(dt)}{
		|t-s|^\alpha}\\
	&= \iint \sum_{l=0}^\infty \frac{p^{-l}}{|G_l|}
		\left( 1-q\left\{
		1-e^{-|t-s|}\right\} \right)^l  \frac{\nu(ds)\, \nu(dt)}{
		|t-s|^\alpha}.
\end{split}\end{equation}
It may appear that this expression is difficult to work with. To illustrate
that this is not always the case we derive the following bound
which may be of independent interest.
Our next result computes the Hausdorff
dimension of $S(G)\cap D$ in case that $D$ is a nice fractal;
i.e., one whose packing and Hausdorff dimensions agree. Throughout,
$\dimp$ denotes packing dimension \cites{Tricot,Sullivan}.

\begin{proposition}\label{pr:dimh:SS}
	Suppose $G$ is an infinite spherically symmetric tree.
	Suppose also 
	that $\P_p(\cup_{t\in D}\{\rho\stackrel{t}{\leftrightarrow}\infty\})>0$ 
	for a certain non-random compact set $D\subseteq\R_+$. If 
	$\delta:=\dimh D$ and $\Delta:=\dimp D$, then $\P_p$-almost surely
	on $\cup_{t\in D}\{\rho\stackrel{t}{\leftrightarrow}\infty\}$,
	\begin{equation}
		\left[ \dimh S(G)
		- (1-\delta) \right]_+ \le
		\dimh \left( S(G)\cap D\right) \le \left[ \dimh S(G)
		- (1-\Delta)\right]_+.
	\end{equation}
\end{proposition}

\begin{remark}
	An anonymous referee has pointed the following
	consequence: If the Hausdorff and packing dimensions
	of $D$ are the same, then a.s.\ on $\{S(G)\cap D\neq\varnothing\}$,
	\begin{equation}
		1-\dimh \left( S(G)\cap D\right)=
		\left( 1- \dimh S(G)\right) + 
		\left( 1 - \dimh D\right).
	\end{equation}
	This agrees with the principle that ``codimensions
	add.''
\end{remark}

\begin{proof}
	Without loss of generality, we may assume that $G$ has no leaves
	[except $\rho$].
	
	The condition
	$\P_p(\cup_{t\in D}\{\rho\stackrel{t}{\leftrightarrow}\infty\})>0$
	and ergodicity together prove that
	there a.s.\ $[\P_p]$ exists a time $t$ of percolation. Therefore,
	\eqref{eq:HPS} implies that
	\begin{equation}\label{eq:summability}
		\sum_{i=1}^\infty \frac{p^{-l}}{l|G_l|}<\infty.
	\end{equation}
	This is in place throughout.
	Next we proceed with the harder lower bound first.
	Without loss of generality, we may assume
	that $\dimh S(G)> 1-\delta$, for otherwise there is nothing
	left to prove.
	
	According to Frostman's lemma, 
	there exists $\nu\in\mathcal{P}(D)$ such that
	for all $\e>0$ we can find a constant $C_\e$ 
	with the following property:
	\begin{equation}\label{eq:frostman:bound}
		\sup_{x\in D} \nu\left( [x-r\,,x+r]\right) \le
		C_\e r^{\delta-\e},\qquad\text{for all}\quad r>0.
	\end{equation}
	\cite{Khoshnevisan}*{Theorem 2.1.1, p.\ 517.}
	We shall fix this $\nu$ throughout the derivation of the lower
	bound.
	
	Choose and fix $\alpha$ that satisfies
	\begin{equation}\label{eq:alpha:bd}
		0< \alpha < \frac{\dimh S(G) -1}{1-\e} + \delta-\e.
	\end{equation}
	[Because we assumed that
	$\dimh S(G)>1-\delta$ the preceding bound is valid
	for all $\e>0$ sufficiently small. Fix such a $\e$
	as well.]
	For this particular $(\nu\,, \alpha\,,\e)$
	we apply to \eqref{eq:I:cross} the elementary bound
	$1-q\{1-e^{-x}\}\le \exp(-qx/2)$,
	valid for all $0\le x\le 1$, and obtain
	\begin{equation}
		I_{\phi(\alpha)} (\mathfrak{m}_{_{\partial G}}\times\nu) \le
		 \sum_{l=0}^\infty \frac{p^{-l}}{|G_l|} \iint
		\exp\left( -\frac{ql|t-s|}{2} \right)  \frac{\nu(ds)\, \nu(dt)}{
		|t-s|^\alpha}.
	\end{equation}
	We split the integral in two parts, according to whether
	or not $|t-s|$ is small, and deduce that
	\begin{equation}
		I_{\phi(\alpha)} (\mathfrak{m}_{_{\partial G}}\times\nu) 
		\le \sum_{l=0}^\infty \frac{p^{-l}}{|G_l|} 
		\iint_{|t-s|\le l^{-(1-\e)}}
		\frac{\nu(ds)\, \nu(dt)}{
		|t-s|^\alpha}
		+\sum_{l=0}^\infty \frac{p^{-l}
		l^{(1-\e)\alpha}e^{-(ql^\e)/2}}{|G_l|}.
	\end{equation}
	Thanks to \eqref{eq:summability} the last
	term is a finite number, which we call $K_\e$. Thus,
	\begin{equation}\begin{split}
		I_{\phi(\alpha)} (\mathfrak{m}_{_{\partial G}}\times\nu)
			&\le \sum_{l=0}^\infty \frac{p^{-l}}{|G_l|} 
			\iint_{|t-s|\le l^{-(1-\e)}}
			 \frac{\nu(ds)\, \nu(dt)}{
			|t-s|^\alpha}+K_\e.
	\end{split}\end{equation}
	Integration by parts shows that if $f:\R\to\R_+\cup\{\infty\}$
	is even, as well
	as right-continuous and non-increasing on $(0\,,\infty)$,
	then for all $0< a<b$, 
	\begin{equation}\label{eq:IBP}
		\iint_{a\le |t-s|\le b}  f(s-t)\,\nu(ds)\,\nu(dt)
		= f(x)F_\nu(x)\Big|_a^b +
		\int_a^b F_\nu(x)\, d|f|(x),
	\end{equation}
	where $F_\nu(x) := (\nu\times\nu)\{(s\,,t)\in\R^2_+:\
	|s-t|\le x\}\le C_\e x^{\delta-\e}$
	thanks to \eqref{eq:frostman:bound}. 
	We apply this bound with $a\downarrow 0$, $b:=l^{-(1-\e)}$,
	and $f(x):=|x|^{-\alpha}$
	to deduce that
	\begin{equation}
		\iint_{|t-s|\le l^{-(1-\e)}}
		\frac{\nu(ds)\, \nu(dt)}{
		|t-s|^\alpha}
		\le A l^{(1-\e)(\alpha-\delta+\e)},
	\end{equation}
	since $\alpha<\delta-\e$ by \eqref{eq:alpha:bd}.
	Here, $A:=C_\epsilon(\delta-\e)/(-\alpha+\delta-\e)$.
	Consequently,
	\begin{equation}
		I_{\phi(\alpha)} (\mathfrak{m}_{_{\partial G}}\times\nu) 
		\le A\sum_{l=0}^\infty \frac{p^{-l}l^{
		(1-\e)(\alpha-\delta+\e)}}{|G_l|}
		+K_\e,
	\end{equation}
	which is finite
	thanks to \eqref{eq:alpha:bd} and \eqref{eq:dimh}.
	It follows from Theorem \ref{th:dimh:SS} that $\P_p$-almost
	surely,
	$\dimh (S(G)\cap D) \ge \alpha$. Let $\e\downarrow 0$
	and $\alpha\uparrow \dimh S(G)-1+\delta$ in \eqref{eq:alpha:bd}
	to obtain the desired lower bound.
		
	Choose and fix $\beta>\Delta$ and
	$\alpha>\dimh S(G)+\beta-1$. We appeal to
	\eqref{eq:I:cross} and the following elementary bound:
	For all integers $l\ge 1$ and all $0\le x\le 1/l$,
	we have
	$(1-q\{1-e^{-x}\})^l\ge p$.
	It follows from this and \eqref{eq:I:cross}
	that for all $\nu\in\mathcal{P}(E)$,
	\begin{equation}\begin{split}
		I_{\phi(\alpha)} (\mathfrak{m}_{_{\partial G}}\times\nu) 
			&\ge p\sum_{l=0}^\infty \frac{p^{-l}}{|G_l|} 
			\iint_{|t-s|\le l^{-1}}
			\frac{\nu(ds)\, \nu(dt)}{
			|t-s|^\alpha}\\
		&\ge p\sum_{l=0}^\infty \frac{p^{-l} l^{\alpha}}{|G_l|} 
			\int \nu\left(\left(
			t-\frac{1}{l} ~,~ t+\frac{1}{l} \right)\right)\, \nu(dt).
	\end{split}\end{equation}
	Because $\beta>\dimp D$, the density theorem of
	Taylor and Tricot \ycite{TaylorTricot}*{Theorem 5.4} implies that
	\begin{equation}\begin{split}
		\liminf_{\e\to 0^+}\frac{1}{\e^\beta}\int \nu\left((t-\e
			~,~ t+\e)\right) \, \nu(dt) 
			&\ge \int \liminf_{\e\to 0^+}\frac{\nu\left((t-\e
			~,~ t+\e)\right)}{\e^\beta}\, \nu(dt)\\
		&=\infty.
	\end{split}\end{equation}
	We have also applied Fatou's lemma. Thus, 
	there exists $c>0$ such that
	for all $\nu\in\mathcal{P}(E)$,
	\begin{equation}\begin{split}
		I_{\phi(\alpha)} (\mathfrak{m}_{_{\partial G}}\times\nu) 
			&\ge c\sum_{l=0}^\infty \frac{p^{-l} l^{\alpha-\beta}}{|G_l|}\\
		&=\infty;
	\end{split}\end{equation}
	see \eqref{eq:dimh}.
	It follows from Theorem \ref{th:dimh:SS} that $\P_p$-almost
	surely,
	$\dimh (S(G)\cap D) \le \alpha$. 
	Let $\beta\downarrow\Delta$ and then $\alpha\downarrow
	\dimh S(G)+\Delta-1$, in this order, to finish.
\end{proof}

\section{\bf On Strong $\beta$-Sets}\label{sec:examples}
An anonymous referee has pointed out
that the abstract capacity condition of Corollary
\ref{co:main} is in general difficult to check.
And we agree. In the previous section we showed how 
such computations can be carried out when $G$ is spherically
symmetric but $D$ is arbitrary. The goal of this section is to provide
similar types of examples in the case that
$G$ is arbitrary but $D$ is a ``nice'' set.

Henceforth we choose and fix some
$\beta\in(0\,,1]$, and 
assume that the target set $D\subset[0\,,1]$ is a
\emph{strong $\beta$-set}; this is a stronger condition
than the better known $s$-set condition of \ocite{Besicovitch}. Specifically, 
we assume that there exists
$\sigma\in\mathcal{P}(D)$ and constant $c_1,c_2\in(0\,,\infty)$, 
such that
\begin{equation}\label{eq:sigma}
	c_1 \e^\beta\le
	\sigma( [x-\e\,,x+\e]) \le c_2 \e^\beta
	\qquad{}^\forall x\in D,\, \e\in(0\,,1).
\end{equation}
We can combine the density theorems of
\ocite{Frostman} together with that of
\ocite{TaylorTricot} to find that
the packing and Hausdorff dimensions of $D$ agree,
and are both equal to $\beta$.
Next we present an amusing example
from symbolic dynamics; other examples abound.

\begin{example}[Cantor Sets]\label{example:Cantor}
	Choose and fix an integer $b\ge 2$.
	We can write any $x\in[0\,,1]$ as
	$x=\sum_{j=1}^\infty b^{-j} x_j$, where
	$x_j\in\{0\,,\ldots,b-1\}$. In case of ambiguities
	we opt for the infinite expansion always. This defines
	the base-$b$ digits $x_1,x_2,\ldots$ of $x$ uniquely.
	Let $B$ denote a fixed subset of $\{0\,,\ldots,b-1\}$,
	and define $D$ to be the closure of
	\begin{equation}
		D_0:= \left\{ x\in[0\,,1]:\ x_j\in B
		\text{ for all }j\ge 1 \right\}.
	\end{equation}
	Then, $D$ is a $\beta$-set  with
	$\beta:={\log_b |B|}$, where
	$\log_b$ denotes the base-$b$ logarithm and
	$|B|$ the cardinality of $B$. Indeed, let $X_1,X_2,\ldots$
	denote i.i.d.\ random variables with uniform distribution on 
	$B$, and observe that for all integers $n\ge 1$ and all $x\in D$,
	\begin{equation}
		\P\left\{ X_1=x_1 \,, \ldots, X_n=x_n\right\}
		= |B|^{-n} = b^{-n\beta}.
	\end{equation}
	If $y\in[0\,,1]$ satisfies $y_1=x_1,\ldots,y_n=x_n$,
	then certainly $|x-y|\le b\sum_{j=n+1}^\infty
	b^{-j}:= M b^{-n}$. Conversely, if $|x-y|\le b^{-n}$
	then it must be the case that
	$y_1=x_1,\ldots,y_n=x_n$. Let $\sigma$ denote the 
	distribution of $X=\sum_{j=1}^\infty b^{-j} X_j$,
	and we know a priori that $\sigma
	\in\mathcal{P}(D)$. In addition
	for all $x\in D$,
	\begin{equation}
		\sigma\left(\left[ x - b^{-n}\,,
		x+ b^{-n}
		\right]\right)
		\le b^{-n\beta}\le \sigma\left( \left[
		x-Mb^{-n}\,, x+Mb^{-n}\right]\right).
	\end{equation}
	A direct monotonicity argument proves the assertion
	that $D$ is a $\beta$-set with 
	$\beta:=\log_b |B|$. We note further that
	if $b:=3$ and $B:=\{0\,,2\}$ then $D$ is nothing more
	than the usual ternary Cantor set in $[0\,,1]$,
	$\sigma$ is the standard Cantor--Lebesgue measure,
	and $\beta=\log_3 2$. For another noteworthy example
	set $B:=\{0\,,\ldots,b-1\}$ to find that $D=[0\,,1]$,
	$\sigma$ is the standard Lebesgue measure on $[0\,,1]$,
	and $\beta:=1$.
\end{example}

\begin{theorem}\label{th:beta-set}
	Suppose $G$ is an arbitrary locally finite tree
	and $D$ is a strong $\beta$-set for some
	$\beta\in(0\,,1]$. 
	Then, $\P_p(\cup_{t\in D}
	\{\rho\stackrel{t}{\leftrightarrow}\infty\})>0$ iff
	$D$ has positive $g$-capacity, where
	\begin{equation}
		g(v\,,w) := \frac{1}{|v\wedge w|^{\beta}\
		p^{|v\wedge w|}}\qquad
		{}^\forall v,w\in\partial G.
	\end{equation}
\end{theorem}

\begin{remark}\label{rem:interpolate}
	It is easy to see that $D:=\{0\}$ is a strong
	$0$-set with $\sigma:=\delta_0$, and
	$D:=[0\,,1]$ is a strong $1$-set with
	$\sigma$ being the Lebesgue measure on $[0\,,1]$;
	see the final sentence in Example
	\ref{example:Cantor}.
	Therefore, it follows that Theorem \ref{th:beta-set}
	contains both Lyons's condition \eqref{eq:Lyons},
	as well as the condition 
	\eqref{eq:HPS} of H\"aggstr\"om et al.\
	as special cases.
\end{remark}

\begin{proof}[Sketch of proof of Theorem \ref{th:beta-set}]
	We can employ a strategy similar to
	the one we used to prove Theorem 
	\ref{th:YuvalConj}, and establish first that $\P_p(\cup_{t\in D}
	\{\rho\stackrel{t}{\leftrightarrow}\infty\})>0$
	if and only if $\text{Cap}_\psi(\partial G)>0$,
	where $\psi(v\,,w):=R(|v\wedge w|)$ and
	\begin{equation}
		R(n) := \iint
		\left( 1 + \frac{q}{p} e^{-|t-s|}\right)^{n}
		\, \sigma(ds)\,\sigma(dt)\qquad{}^\forall n\ge 1.
	\end{equation}
	It might help to recall that $\sigma$
	was defined in \eqref{eq:sigma}. We integrate by parts
	as in \eqref{eq:IBP} and then apply \eqref{eq:sigma}
	to find that
	\begin{equation}
		a_1 n \int_0^1
		x^\beta\left( 1 + \frac{q}{p}e^{-x}\right)^{n-1}\, dx
		\le R(n) \le
		a_2 n \int_0^1
		x^\beta\left( 1 + \frac{q}{p}e^{-x}\right)^{n-1}\, dx,
	\end{equation}
	where the $a_i$'s are positive and finite
	constants that do not depend on $n$. From here
	it is possible to check that
	$R(n)$ is bounded above and below
	by constant multiples of $n^{-\beta} p^{-n}$. Indeed,
	we use $\int_0^1(\cdots)\ge\int_{1/n}^1(\cdots)$
	to obtain a lower bound. For an upper bound, we
	decompose $\int_0^1(\cdots)$ as $\int_0^{1/n}(\cdots)+
	\int_{1/n}^{1/2}(\cdots) + \int_{1/2}^1(\cdots)$,
	and verify that the first integral dominates the other
	two for all large values of $n$. This has the desired effect.
\end{proof}

\end{document}